\documentclass[a4paper,twoside]{amsart}
\usepackage[T1]{fontenc}
\usepackage[utf8]{inputenc}

\usepackage{hyperref}
\usepackage[slantedGreeks, partialup, noDcommand]{kpfonts}
\usepackage{graphicx}

\hypersetup{ocgcolorlinks=true,allcolors=testc}
\hypersetup{
     colorlinks   = true,
     citecolor    = black
}
\hypersetup{linkcolor=black}
\hypersetup{urlcolor=black}

\newtheorem{theorem}{Theorem}
\newtheorem{lemma}[theorem]{Lemma}
\newtheorem{corollary}[theorem]{Corollary}

\newtheorem{remark}[theorem]{Remark}

\newtheorem*{definition*}{Definition}

\title[Sharp growth of the Ornstein--Uhlenbeck operator]{Sharp growth of the Ornstein--Uhlenbeck operator\\ on Gaussian tail spaces}  

\author{Alexandros Eskenazis}
\address{A. Eskenazis, Trinity College and Department of Pure Mathematics and Mathematical Statistics\\
University of Cambridge, UK.}
\email{ae466@cam.ac.uk}

\author{Paata Ivanisvili}
\address{P. Ivanisvili,
Department of Mathematics,
North Carolina State University\\
Raleigh, NC 27695.}
\email{pivanis@ncsu.edu} 

\thanks{A.~E.~was supported by a Junior Research Fellowship from Trinity College, Cambridge. P.~I.~ was supported in part by NSF grants DMS-1856486, DMS-2052645,  and CAREER DMS-2052865}

\begin{document}

\maketitle
\vspace{-3mm}

\begin{abstract}
Let $X$ be a standard Gaussian random variable. For any $p \in (1, \infty)$, we prove the existence of a universal constant $C_{p}>0$ such that the inequality
 $$
 (\mathbb{E} |h'(X)|^{p})^{1/p} \geq C_{p} \sqrt{d} (\mathbb{E} |h(X)|^{p})^{1/p}
 $$ 
 holds for all $d\geq 1$ and all polynomials $h : \mathbb{R} \to \mathbb{C}$ whose spectrum is supported on frequencies at least $d$, that is, $\mathbb{E} h(X) X^{k}=0$ for all $k=0,1, \ldots, d-1$.  As an application of this optimal estimate, we obtain an affirmative answer to the Gaussian analogue of a question of Mendel and Naor (2014) concerning the growth of the Ornstein--Uhlenbeck operator on tail spaces of the real line. We also show the same bound for the gradient of analytic polynomials in an arbitrary dimension. 
\end{abstract}

\bigskip

{\footnotesize

}

\section{Introduction}
\subsection{Approximation theory on the hypercube and the Gauss space}

Fix $n\geq 1$, and let $\{-1,1\}^{n}$ be the hypercube of dimension $n$. Any function $f : \{-1,1\}^{n} \to \mathbb{C}$ can be decomposed into Fourier--Walsh series as
\begin{align*}
f(x) = \sum_{S \subseteq \{1, \ldots, n \}} a_{S} \prod_{j \in S} x_{j}
\end{align*}
where $x = (x_{1}, \ldots, x_{n}) \in \{-1,1\}^{n}$ and $a_{S} \in \mathbb{C}$ are the Fourier--Walsh coefficients. The hypercube Laplacian $\Delta$ of $f$ is defined as 
\begin{align*}
\Delta f = \sum_{j=1}^{n} D_{j}f, \quad \text{where} \quad D_{j}f(x) = \frac{f(x_{1}, \ldots, x_{j}, \ldots, x_{n}) -f(x_{1}, \ldots, -x_{j}, \ldots, x_{n}) }{2}
\end{align*}
and the discrete $L_p$ norm is given by
$$
\| f\|_{p} =\Big( \frac{1}{2^{n}} \sum_{x \in \{-1,1\}^{n}} |f(x)|^{p}\Big)^{1/p}.
$$

In their study of  super-expanders, Mendel and Naor~\cite{MN1} asked the following influential question\footnote{In fact, the original question of \cite{MN1} concerned vector-valued functions but we will restrict ourselves to the scalar-valued case for the ensuing discussion.}. Let $1<p<\infty$. Does there exist a universal constant $C_{p}>0$ such that for any $n, d \geq 0$ with $0\leq d \leq n$, we have 
\begin{align}\label{conj1}
\| \Delta f \|_{p} \geq C_{p} d \|f\|_{p}
\end{align}
for all $f : \{-1,1\}^{n} \to \mathbb{C}$ in the {\em $d$-th tail space}, i.e. all functions of the form
$$f(x) = \sum_{\substack{S \subseteq \{1, \ldots, n\} \\ |S|\geq d}} a_{S} \prod_{j \in S} x_{j}\ ?$$

This problem led to a study of a more general question, which is known as the {\em heat smoothing conjecture}, see \cite{OLEG}. Currently, the best known bounds towards \eqref{conj1} are of the form $\| \Delta f\|_{p} \geq C_{p} d^{\alpha_{p}} \| f\|_{p}$ for a certain  $\alpha_{p} \in [1/2,1)$, see \cite{MN1,EI2}.

A standard limiting argument  \cite[pp.~164-166]{Bec75} (see equation (5) there) shows that \eqref{conj1} implies the corresponding estimate in Gauss space
\begin{align}\label{conj2}
\| L h\|_{L^{p}(d\gamma_{n})} \geq C_{p} d \|h\|_{L^{p}(d\gamma_{n})}
\end{align}
for all $n\geq 1$ and all polynomials $h : \mathbb{R}^{n} \to \mathbb{C}$ in the corresponding Gaussian $d$-th tail space; here $L$ is the generator of the Ornstein--Uhlenbeck semigroup, see Section~\ref{mres} for the definitions of these notions. The purpose of the present paper is to investigate the Gaussian counterpart \eqref{conj2} of the discrete inequality \eqref{conj1}, in view of the additional continuous tools which have been developped in the weighted approximation theory literature. In particular, we shall present a proof of \eqref{conj2} for $n=1$ and a proof of \eqref{conj2} for {\em analytic} polynomials in all (even) dimensions.
%As a proof (or a counterexample) of inequality \eqref{conj1} does not seem to be feasible with current techniques perhaps a good starting point is to verify continuous, gaussian counterpart of (\ref{conj1}), namely, the estimate (\ref{conj2}). 

\subsection{Definitions and statement of the main results}\label{mres}
For an integer $n \geq 1$, let 
$$
d\gamma_{n}(x) \stackrel{\mathrm{def}}{=} \frac{e^{-\frac{|x|^{2}}{2}}}{(\sqrt{2\pi})^{n}}dx
$$
be the standard Gaussian measure on $\mathbb{R}^{n}$,  where we denote\footnote{If $w=(w_{1}, \ldots, w_{n}) \in \mathbb{C}^{n}$ we also set $|w| = \sqrt{|w_{1}|^{2}+\ldots+|w_{n}|^{2}}$, where $|z|$ denotes the absolute value of a complex number $z \in \mathbb{C}$, and this notation is consistent with the usual identification of $\mathbb{C}^n$ with $\mathbb{R}^{2n}$.} by $|x| \stackrel{\mathrm{def}}{=} \sqrt{x_{1}^{2}+\ldots+x_{n}^{2}}$ the Euclidean norm of a vector $x = (x_{1}, \ldots, x_{n}) \in \mathbb{R}^{n}$. Let $\alpha = (\alpha_{1}, \ldots, \alpha_{n}) \in (\mathbb{Z}_{+})^{n}$ be a multi-index, where $\mathbb{Z}_{+}$ denotes the set of nonnegative integers, and set $|\alpha| \stackrel{\mathrm{def}}{=} \alpha_{1}+\ldots+\alpha_{n}$. The Hermite polynomial of degree $\alpha$ is defined as 
\begin{align*}
H_{\alpha}(x)  \stackrel{\mathrm{def}}{=}  \int_{\mathbb{R}^{n}}  (x+iy)^{\alpha} d\gamma_{n}(y) \quad \text{where} \quad (x+iy)^{\alpha} \stackrel{\mathrm{def}}{=} \prod_{j=1}^{n} (x_{j}+iy_{j})^{\alpha_{j}}.
\end{align*}
For a finite $p \geq 1$, consider the space 
 \begin{align*}
 L^{p}(\mathbb{R}^{n}, d\gamma_{n}) = \Bigl\{ f \in \mathcal{L}(\mathbb{R}^{n}; \mathbb{C})\,: \,  \; \; \; \|f\|^{p}_{L^{p}(d\gamma_{n})} = \int_{\mathbb{R}^{n}} |f|^{p} d\gamma_{n} < \infty\Bigr\}, 
 \end{align*}
 where we denote by $\mathcal{L}(\mathbb{R}^{n}; \mathbb{C})$ the space of all functions  $f :\mathbb{R}^{n} \to \mathbb{C}$ which are Lebesgue measurable. When $p=\infty$ we simply set $L^{\infty}(\mathbb{R}^{n}; d\gamma_{n})$ to be the space of complex-valued measurable functions such that $\|f\|_{L^{\infty}(d\gamma_{n})} \stackrel{\mathrm{def}}{=}\mathrm{ess sup}\, |f|<\infty$. 
The Hermite polynomials $\{ H_{\alpha}\}_{\alpha \in (\mathbb{Z}_{+})^{n}}$ form an orthogonal system in $L^{2}(d\gamma_{n})$ with respect to the inner product 
 \begin{align*}
 \langle f, g \rangle \stackrel{\mathrm{def}}{=} \int_{\mathbb{R}^{n}} f \overline{g} d\gamma_{n} \quad \text{for all} \quad f, g \in L^{2}(d\gamma_n). 
 \end{align*}
 Orthogonality along with the fact $\mathrm{deg}(H_{\alpha}) \leq |\alpha|$ imply that  
 \begin{align*}
 \mathcal{P}^{\leq d} \stackrel{\mathrm{def}}{=}\mathrm{span}_{\mathbb{C}}\big\{ x^{\alpha}\, :\, |\alpha|\leq d\big\} =  \mathrm{span}_{\mathbb{C}}\big\{ H_{\alpha}(x) \, :\, |\alpha|\leq d\big\}
 \end{align*}
 for any integer $d\geq 0$. Hence, the system $\{H_{\alpha}\}_{\alpha \in (\mathbb{Z}_{+})^{n}}$ is complete in $L^{p}(d\gamma_{n})$ for all $1\leq p <\infty$. Our main space of investigation will be the (finite) $d$-tail space
 \begin{align*}
 \mathcal{P}^{\geq d} \stackrel{\mathrm{def}}{=} \left\{f \in \mathcal{L}(\mathbb{R}^{n}; \mathbb{C})\, :\, f = \sum_{ |\alpha|\geq d} c_{\alpha} H_{\alpha}\; \;  \text{is a finite sum with}\, \; \;  c_{\alpha} \in \mathbb{C} \right\}.
 \end{align*}

 Let $\Delta  = \sum_{j=1}^{n} \partial^{2}_{x_{j}}$ be the usual Laplacian on $\mathbb{R}^n$ and  $\nabla g = (\partial_{x_{1}}g, \ldots, \partial_{x_{n}} g)$ be the gradient of $g$.   We denote by $L = -(\Delta -x\cdot \nabla)$ the (positive) generator  of the Ornstein--Uhlenbeck semigroup, where $x\cdot \nabla  = \sum_{j=1}^{n} x_{j} \partial_{x_{j}}$. Recall that $L H_{\alpha} = |\alpha| H_{\alpha}$. We shall first address the analogue of \eqref{conj2} for the gradient $\nabla$ instead of $L$.
 
 \begin{definition*}
 Fix $p \in (1, \infty)$ and $n \geq 1$. Let $F(n,p)$ be the  smallest  constant such that 
 \begin{align}\label{rfreud}
  \sqrt{d+1}\|f\|_{L^{p}(d\gamma_{n})} \leq F(n,p)\,  \| \nabla f\|_{L^{p}(d\gamma_{n})} 
   \end{align}
 holds for all $d\geq0$ and $f \in \mathcal{P}^{\geq d+1}$. Similarly, let $J(n,p)$ be the smallest constant\footnote{The letter $J$ in the definition of the constant $J(n.p)$ is referring to Jackson \cite{Jac12}, who first investigated inequalities in the spirit of \eqref{jackson} for trigonometric polynomials on the torus $\mathbb{T}$.} such that 
 \begin{align}\label{jackson}
  \inf_{\varphi \in \mathcal{P}^{\leq d}} \| g-\varphi\|_{L^{p}(d\gamma_{n})} \leq \frac{J(n,p)}{\sqrt{d}} \|\nabla g\|_{L^{p}(d\gamma_{n})}
 \end{align}
 holds for all $d\geq1$ and all polynomials $g :\mathbb{R}^{n} \to \mathbb{C}$. 
 \end{definition*}
 
 Our first result is a dimension-free equivalence of the constants in \eqref{rfreud} and \eqref{jackson}.
 \begin{theorem}\label{thm1}
For any  $p \in (1, \infty)$, there exist finite positive constants $C_{p}, c_{p}$ such that 
\begin{align}
c_{p} F(n,p')  \leq J(n,p) \leq C_{p} F(n,p')
\end{align}
 holds for any $n \geq 1$, where $p' \stackrel{\mathrm{def}}{=} \frac{p}{p-1}$ is the dual exponent to $p$. 
 \end{theorem}
 
 Combining Theorem \ref{thm1} with a classical inequality of Freud \cite{Fr1} from weighted approximation theory, we deduce the following consequences.
 
 \begin{corollary}\label{sol1}
 For any $p \in (1, \infty)$, there exist $S_{p}, T_p<\infty$ such that the estimates
 \begin{align}\label{naor}
\sqrt{d}\| f\|_{L^{p}(d\gamma_{1})} \leq S_{p} \|f'\|_{L^{p}(d\gamma_{1})}  \quad \mbox{and} \quad d \| f\|_{L^{p}(d\gamma_{1})} \leq T_{p} \|Lf\|_{L^{p}(d\gamma_{1})}
 \end{align}
 hold for all $f \in \mathcal{P}^{\geq d}$ and any $d\geq 1$. 
 \end{corollary}
 
The first inequality in \eqref{naor} is the converse to Freud's Bernstein--Markov inequality with Gaussian weight \cite{freud01} and the second inequality answers the Gaussian analogue of Mendel and Naor's question on the real line. In the next theorem we show that the same bounds hold true for analytic polynomials in an arbitrary dimension, in particular resolving the problem of \cite{MN1} for such polynomials.
 
 \begin{theorem}\label{anal1}
 For any $n\geq 1$,  $p\geq 1$ and any analytic $f : \mathbb{C}^{n} \to \mathbb{C}$ with $f \in \mathcal{P}^{\geq d}$, we have 
 \begin{align}\label{naoran}
 d \| f \|_{L^{p}(d\gamma_{2n})} \leq \| L f\|_{L^{p}(d\gamma_{2n})},
 \end{align}
 where $d\gamma_{2n}$ is the standard gaussian measure on $ \mathbb{C}^{n}\cong\mathbb{R}^{2n}$. Moreover, for any \mbox{$p\in(1,\infty)$}, there exists a finite constant $K_p$ such that for any $n\geq1$, we have 
 \begin{align}\label{gradan}
 \sqrt{d} \|f\|_{L_p(d\gamma_{2n})} \leq K_p \|\nabla f\|_{L_p(d\gamma_{2n})}
 \end{align}
 for any analytic polynomial $f : \mathbb{C}^{n} \to \mathbb{C}$ with $f \in \mathcal{P}^{\geq d}$.
 \end{theorem}
 We remark that the dependence on $d$ of the estimates appearing in Corollary~\ref{sol1} and Theorem~\ref{anal1} are asymptotically optimal, as can be seen by choosing the Hermite polynomial $f(x)=H_d(x)$ in \eqref{naor} and $f(z)=z_1^d$ in \eqref{naoran} and \eqref{gradan}.
 
Using similar techniques as in the proof of Theorem \ref{anal1} we obtain sharp dimension-free Bernstein--Markov inequalities for analytic polynomials in Gauss space.

\begin{theorem}\label{analgrad}
 For any $p \in (1, \infty)$, there exists a finite constant $C_p$ such that for any $n\geq1$, we have
 \begin{align}\label{naoran2}
 \| \nabla f \|_{L^{p}(d\gamma_{2n})} \leq C_{p} \sqrt{d} \| f\|_{L^{p}(d\gamma_{2n})}
 \end{align}
for any analytic polynomial $f : \mathbb{C}^{n} \to \mathbb{C}$ with $f \in \mathcal{P}^{\leq d}$.
\end{theorem}
The estimate \eqref{naoran2} is asymptotically sharp as can be seen by choosing $f(z)=z_1^d$. Theorem~\ref{naoran2}  extends Freud's inequality~\cite{freud01} to higher dimensions with the sharp dependence on $d$  for analytic polynomials. We refer the interested reader to \cite{EI1}, where  the dimension independent estimate $\| \nabla f\|_{L^{p}(d\gamma_{n})} \leq C_{p} d^{\alpha_{p}} \|f\|_{L^{p}(d\gamma_{n})}$ was obtained with $\alpha_{p} \in [1,2)$ and all (not just \mbox{analytic) polynomials $f$ of degree at most $d$.}

Finally, we derive the following sharp moment comparison result for analytic polynomials in Gauss space (see also Remark \ref{rem:mom}).

\begin{theorem} \label{thm:mom}
For any $0<p<q<\infty$ with $q\geq1$, $n\geq1$ and $d\geq1$, we have
\begin{equation} \label{eq:mom}
\| f\|_{L^{q}(d\gamma_{2n})} \leq \left(\frac{q}{p}\right)^{d/2} \| f\|_{L^{p}(d\gamma_{2n})}
\end{equation}
for any analytic polynomial $f : \mathbb{C}^{n} \to \mathbb{C}$ with $f \in \mathcal{P}^{\leq d}$.
\end{theorem}
 
\section{Proofs}
\subsection{Proof of Theorem~\ref{thm1}}
Fix $p\in(1,\infty)$. First we show the existence of a $c_{p}>0$ such that $c_{p} F(n,p') \leq J(p,n)$. By the Hahn--Banach theorem, for every $g\in L^p(d\gamma_n)$,
\begin{align} \label{hahnba} 
\inf_{\varphi \in \mathcal{P}^{\leq d}} \| g-\varphi\|_{L^{p}(d\gamma_{n})} = \sup_{\substack{f \in (\mathcal{P}^{\leq d})^{\perp}\\ \|f\|_{L^{p'}(d\gamma_{n})} \leq 1}}  |\langle g, f \rangle|,
\end{align}
where the annihilator of the subspace $\mathcal{P}^{\leq d} \subseteq L^{p}(d\gamma_{n})$ is given by 
\begin{align*}
(\mathcal{P}^{\leq d})^{\perp} \stackrel{\mathrm{def}}{=} \big\{ f \in L^{p'}(d\gamma_{n})\; :\; \langle f,\varphi \rangle =0 \quad \text{for all} \quad \varphi \in \mathcal{P}^{\leq d}\big\}.
\end{align*}
Clearly $\mathcal{P}^{\geq d+1} \subseteq (\mathcal{P}^{\leq d})^{\perp}$. Thus, the estimate \eqref{jackson} implies the validity of 
\begin{align}\label{firstred}
 |\langle g, f \rangle|  \leq \frac{J(n,p)}{\sqrt{d}} \|\nabla g\|_{L^{p}(d\gamma_{n})} \|f\|_{L^{p'}(d\gamma_{n})} 
\end{align}
for all polynomials $g$ and all $f \in \mathcal{P}^{\geq d+1}$. By Meyer's two-sided dimension-free bounds for the Riesz transform in Gauss space (see~\cite{Me1}),  there exist two finite positive constants $m_{p}$ and  $M_{p}$ such that 
\begin{align}\label{riesz}
m_{p} \| L^{1/2} g \|_{L^{p}(d\gamma_{n})}  \leq \| \nabla g \|_{L^{p}(d\gamma_{n})} \leq M_{p}\| L^{1/2} g \|_{L^{p}(d\gamma_{n})},
\end{align}
where for a polynomial $g = \sum c_{\alpha} H_{\alpha}$ we have $L^{1/2} g = \sum |\alpha|^{1/2} c_{\alpha} H_{\alpha}$. 
Combining \eqref{riesz} with \eqref{firstred} we obtain the inequality
\begin{align}\label{sred}
 |\langle g, f \rangle|  \leq \frac{M_{p}J(n,p)}{\sqrt{d}} \| L^{1/2} g\|_{L^{p}(d\gamma_{n})} \|f\|_{L^{p'}(d\gamma_{n})}
\end{align}
for all polynomials $g$ and all $f \in \mathcal{P}^{\geq d+1}$. Equivalently, we can rewrite \eqref{sred} as 
\begin{align}\label{3red}
|\langle L^{-1/2}(g-\mathbb{E}g), f \rangle|  \leq \frac{M_{p}J(n,p)}{\sqrt{d}} \| g - \mathbb{E} g\|_{L^{p}(d\gamma_{n})} \|f\|_{L^{p'}(d\gamma_{n})},
\end{align}
where $\mathbb{E}g = \int_{\mathbb{R}^{n}} g d\gamma_{n}$ and $L^{-1/2} h = \sum_{|\alpha|>0} |\alpha|^{-1/2}\beta_{\alpha} H_{\alpha}$ for a polynomial  of the form $h = \sum_{|\alpha|>0} \beta_{\alpha} H_{\alpha}$. Using the fact $L^{-1/2} \mathcal{P}^{\geq d+1} = \mathcal{P}^{\geq d+1}$, the identity 
$\langle L^{-1/2}(g-\mathbb{E}g), f \rangle  = \langle g, L^{-1/2}f\rangle$,  and the  inequality $\| g-\mathbb{E}g\|_{L^{p}(d\gamma_{n})} \leq 2 \| g\|_{L^{p}(d\gamma_{n})}$,  we see that \eqref{3red} implies
\begin{align}\label{4red}
\begin{split}
|\langle g, f \rangle|  \leq \frac{2 J(n,p)M_{p}}{\sqrt{d}} &\|g\|_{L^{p}(d\gamma_{n})}  \| L^{1/2} f\|_{L^{p'}(d\gamma_{n})} 
\\ & \stackrel{\eqref{riesz}}{\leq} 
\frac{2 J(n,p)M_{p}}{m_{p'}  \sqrt{d}} \|g\|_{L^{p}(d\gamma_{n})} \| \nabla f\|_{L^{p'}(d\gamma_{n})}.
\end{split}
\end{align}
Notice that polynomials are dense in $L^{p}(d\gamma_{n})$. Indeed, by a simple iteration argument and the H\"older inequality  it is enough to study the case $n=1$. By Hahn--Banach it suffices to show that if $h \in L^{p'}(d\gamma_{1})$ is such that $\int_{\mathbb{R}} h(x)x^{k} d\gamma_{1}(x)=0$ for all $k=0,1,\ldots,$ then $h=0$. The latter follows from the fact that the Fourier transform $\widehat{w}$ of $w(x) = h(x)e^{-x^{2}/2}$ is an entire function on $\mathbb{C}$ with $\widehat{w}^{(k)}(0)=0$ for all $k \in \mathbb{Z}_{+}$.

To finish the proof of the first inequality in Theorem~\ref{thm1}, we take the supremum over all polynomials $g \in L^{p}(d\gamma_{n})$ in \eqref{4red}, use the fact that polynomials are dense in $L^{p}(d\gamma_{n})$, and the trivial inequality $\sqrt{2d} \geq \sqrt{d+1}$ to obtain 
\begin{align*}
\sqrt{d+1} \| f\|_{L^{p'}(d\gamma_{n})} \leq  \frac{J(n,p) 2\sqrt{2}M_{p}}{m_{p'}} \|\nabla f\|_{L^{p'}(d\gamma_{n})}.
\end{align*}
Thus  $c_{p} F(n,p') \leq  J(n,p)$ with $c_{p} = \frac{1}{\sqrt{8}}\frac{m_{p'}}{M_{p}}$. 

We will now show the second inequality in Theorem~\ref{thm1}, i.e., $J(n,p)\leq C_{p} F(n, p')$. Using \eqref{rfreud} and Meyer's inequalities \eqref{riesz}, we see that 
\begin{align}\label{conv1}
\sqrt{d+1} \|f\|_{L^{p'}(d\gamma_{n})} \leq F(n,p') \|\nabla f\|_{L^{p'}(d\gamma_{n})} \leq M_{p'} F(n,p') \|L^{1/2} f\|_{L^{p'}(d\gamma_{n})} 
\end{align}
holds for all $f \in \mathcal{P}^{\geq d+1}$. Equivalently, \eqref{conv1} asserts that the bound
\begin{align*}
|\langle g, L^{-1/2} f \rangle| \leq \frac{M_{p'} F(n,p')}{\sqrt{d+1}} \| f\|_{L^{p'}(d\gamma_{n})} \|g\|_{L^{p}(d\gamma_{n})}
\end{align*}
holds for all polynomials $g$ and all $f \in \mathcal{P}^{\geq d+1}$. Using the identity $\langle g,L^{-1/2}f \rangle  = \langle L^{-1/2}(g-\mathbb{E}g), f\rangle$, denoting $h = L^{-1/2}(g-\mathbb{E}g) \in \mathcal{P}^{\geq 1}$ and using Meyer's estimates \eqref{riesz} again, we obtain 
\begin{align}\label{red5}
|\langle h, f \rangle| \leq \frac{M_{p'} F(n,p')}{m_{p} \sqrt{d+1}}  \| f\|_{L^{p'}(d\gamma_{n})}  \| \nabla h\|_{L^{p}(d\gamma_{n})}
\end{align}
for all $f \in \mathcal{P}^{\geq d+1}$ and all polynomials $h$ (since $\nabla h = \nabla (h+\mathbb{E} h)$ and $\langle h,f\rangle = \langle h+\mathbb{E}h,f\rangle$). We will need the following lemma.

\begin{lemma}\label{clos}
Let $p\in (1, \infty)$, $d\geq1$ and consider $\mathcal{P}^{\leq d}$ as a subspace of $L^{p}(d\gamma_{n})$. Then, $ (\mathcal{P}^{\leq d})^{\perp} = \mathrm{cl}_{L^{p'}(d\gamma_{n})}(\mathcal{P}^{\geq d+1})$, where $\mathrm{cl}_{L^{p'}(d\gamma_n)}$ denotes the closure in $L^{p'}(d\gamma_{n})$.
\end{lemma}
Let us first explain why Lemma \ref{clos} combined with \eqref{red5} implies $J(n,p) \leq C_{p} F(n,p')$. Indeed, for all polynomials $h$, we have 
\begin{align*}
\inf_{\varphi \in \mathcal{P}^{\leq d}} \| h -\varphi\|_{L^{p}(d\gamma_{n})} \stackrel{\eqref{hahnba}}{=} \sup_{\substack{f \in (\mathcal{P}^{\leq d})^{\perp}\\ \|f\|_{L^{p'}(d\gamma_{n})} \leq 1}}  |\langle h, f \rangle|  = \sup_{\substack{f \in \mathcal{P}^{\geq d+1}\\ \|f\|_{L^{p'}(d\gamma_{n})} \leq 1}} &  |\langle h, f \rangle| 
\\ & \stackrel{\eqref{red5}}{\leq}  \frac{M_{p'} F(n,p')}{m_{p} \sqrt{d+1}}   \| \nabla h\|_{L^{p}(d\gamma_{n})},
\end{align*}
where in the second equality we used Lemma \ref{clos}. This implies that $J(n,p) \leq C_{p} F(n,p')$ with $C_{p} = \frac{M_{p'}}{m_{p}}$ and Theorem \ref{thm1} follows from the lemma. \hfill$\Box$

\begin{proof}[Proof of Lemma \ref{clos}]
The inclusion $\mathrm{cl}_{L^{p'}(d\gamma_{n})}(\mathcal{P}^{\geq d+1}) \subseteq (\mathcal{P}^{\leq d})^{\perp}$ is trivial and follows from $\mathcal{P}^{\geq d+1} \subseteq (\mathcal{P}^{\leq d})^{\perp}$ and the fact that $(\mathcal{P}^{\leq d})^{\perp}$ is a closed subspace of $L^{p'}(d\gamma_{n})$.

 To verify the reverse inclusion $ (\mathcal{P}^{\leq d})^{\perp} \subseteq \mathrm{cl}_{L^{p'}(d\gamma_{n})}(\mathcal{P}^{\geq d+1})$  assume the contrary, i.e.,  that there exists $f \in (\mathcal{P}^{\leq d})^{\perp} \subseteq L^{p'}(d\gamma_{n})$  and $ f \notin  \mathrm{cl}_{L^{p'}(d\gamma_{n})}(\mathcal{P}^{\geq d+1})$. By the Hahn--Banach theorem, there exists a nonzero element $h \in L^{p}(\gamma_{n})$ such that  $\langle h, f\rangle >0$, and $\langle h, \varphi \rangle =0$ for all  $\varphi \in \mathrm{cl}_{L^{p'}(d\gamma_{n})}(\mathcal{P}^{\geq d+1})$. Define 
 \begin{align*}
 S_{d}(h) \stackrel{\mathrm{def}}{=} \sum_{|\alpha|\leq d} c_{\alpha} H_{\alpha}, \quad \text{where} \quad c_{\alpha} = \frac{\langle h, H_{\alpha} \rangle}{\langle H_{\alpha}, H_{\alpha} \rangle }
 \end{align*}
 and set $\tilde{h} = h - S_{d}(h)$. Clearly $\langle \tilde{h}, f \rangle = \langle h, f \rangle >0$. On the other hand we have 
 $\langle \tilde{h}, H_{\alpha} \rangle =0$  for all $|\alpha| \leq d$ which is the same as $\langle \tilde{h}, \psi \rangle =0$ for all $\psi \in \mathcal{P}^{\leq d}$. Also notice that $\langle \tilde{h}, \varphi \rangle  = \langle h, \varphi \rangle =0$ for all $\varphi \in \mathcal{P}^{\geq d+1}$. It follows that  $\langle \tilde{h}, g \rangle =0 $ for all polynomials $g$. Thus $\tilde{h}=0$ almost surely and this contradicts the fact that $\langle \tilde{h}, f \rangle >0$.
 \end{proof}

\begin{remark}
A version of the equivalence  
\begin{align}\label{eq22}
c_{p} F^{h}(n,p') \stackrel{p\geq 2}{\leq} J^{h}(n,p) \stackrel{1<p\leq 2}{\leq } C_{p} F^{h}(n,p')
\end{align}
also holds on the Hamming cube $\{-1,1\}^{n}$, where the constants $F^{h}, J^{h}$ are defined in the corresponding way as in \eqref{rfreud}-\eqref{jackson} and $|\nabla f| = \sqrt{\sum_{j=1}^{n} |D_{j} f|^{2}}$ for $f : \{-1,1\}^{n} \to \mathbb{C}$. The proofs proceed verbatim. The only difference is that the first inequality in \eqref{eq22} holds only for $p \geq 2$, and the second inequality for $p \in (1, 2]$ due to the fact that the Riesz transform on the Hamming cube is bounded (the analogue of the  second inequality in \eqref{riesz}) only for $p\geq 2$ (see \cite{LP98}). 

\end{remark}

\subsection{Proof of Corollary~\ref{sol1}}
The case $d=1$ is simply the Poincar\'e inequality, so we will assume that $d\geq2$. In \cite{Fr1}, Freud obtained an analogue of Jackson's inequality with weight $e^{-x^{2}/2}$ (see also \cite[Sections~3-4]{Lub07}), that is, he showed that for every $d\geq2$ the estimate
\begin{align}\label{ffr}
\inf_{\varphi \in \mathcal{P}^{\leq d-1}_{\mathbb{R}}} \left(\int_{\mathbb{R}} |(f(x) - \varphi(x))e^{-x^{2}/2}|^{q}dx \right)^{1/q} \leq \frac{C}{\sqrt{d-1}}\left(\int_{\mathbb{R}} |f'(x) e^{-x^{2}/2}|^{q} dx \right)^{1/q}
\end{align} 
holds for all polynomials $f : \mathbb{R} \to \mathbb{R}$, and all  $q \in [1, \infty]$, where $C$ is a universal constant. Here $\mathcal{P}^{\leq d-1}_{\mathbb{R}}$ denotes the space of real-valued polynomials of degree at most $d-1$.  
Rescaling the variables we rewrite \eqref{ffr} as 
\begin{align}\label{jack22}
\inf_{\varphi \in \mathcal{P}^{\leq d-1}_{\mathbb{R}}} \| f - \varphi\|_{L^{q}(d\gamma_{1})} \leq \frac{C\sqrt{q}}{\sqrt{d-1}} \| \nabla f\|_{L^{q}(d\gamma_{1})}
\end{align}

Estimate \eqref{jack22} easily extends to complex-valued polynomials $f$ provided that the infimum in  the left hand side is taken over complex-valued polynomials $\varphi$ of degree at most $d-1$. Indeed, if $f : \mathbb{R} \to \mathbb{C}$ is complex-valued polynomial then we can decompose $f = f_{1}+if_{2}$ where $f_{1}, f_{2}$ are real-valued polynomials. Finally, applying \eqref{jack22} to each $f_{1}$ and $f_{2}$ separately and using the triangle inequality, we get 
\begin{align*}
\inf_{\varphi \in \mathcal{P}^{\leq d-1}} \| (f_{1}+if_{2}) - &\varphi\|_{L^{q}(d\gamma_{1})}  \leq \inf_{\varphi \in \mathcal{P}^{\leq d-1}_{\mathbb{R}}} \| f_{1} - \varphi\|_{L^{q}(d\gamma_{1})} +\inf_{\varphi \in \mathcal{P}^{\leq d-1}_{\mathbb{R}}} \| f_{2} - \varphi\|_{L^{q}(d\gamma_{1})} 
\\ & \leq \frac{C\sqrt{q}}{\sqrt{d-1}} ( \| \nabla f_{1}\|_{L^{q}(d\gamma_{1})}+ \| \nabla f_{2}\|_{L^{q}(d\gamma_{1})}) \leq \frac{2 C\sqrt{q}}{\sqrt{d-1}} \| \nabla(f_{1}+if_{2})\|_{L^{q}(d\gamma_{1})},
\end{align*}
which is equivalent to $J(n,q) \leq 2C\sqrt{q}$. Applying Theorem~\ref{thm1} with $q \in (1, \infty)$ we deduce that there exists a finite constant $B_q$ such that 
\begin{align}\label{app11}
\sqrt{d} \| g\|_{L^{q'}(d\gamma_{1})} \leq B_{q}\| \nabla g\|_{L^{q'}(d\gamma_{1})}
\end{align}
holds for all $g \in \mathcal{P}^{\geq d}$, where $q'=\frac{q}{q-1}$ is the conjugate exponent to $q$. This proves the first inequality of \eqref{naor}. Using Meyer's estimate \eqref{riesz}, \eqref{app11} implies that  $$\sqrt{d} \|g\|_{L^{q'}(d\gamma_{1})} \leq B_{q}M_{q'} \|L^{1/2} g\|_{L^{q'}(d\gamma_{1})}.$$ Finally, iterating the last inequality we obtain 
$$d \|g\|_{L^{q'}(d\gamma_{1})} \leq (B_{q} M_{q'})^{2}\| Lg\|_{L^{q'}(d\gamma_{1})}$$
and this proves the second inequality of \eqref{naor}. \hfill$\Box$

%\begin{remark}
%One can slightly improve the constant $B_{q}$ in (\ref{app11}). Indeed,  following the proof of Theorem~\ref{thm1} one can deduce $\sqrt{d} \|g\|_{L^{q'}(d\gamma_{n})} \leq C\sqrt{8}\sqrt{q}M_{q} \| L^{1/2} g\|_{L^{q'}(d\gamma_{1})}$ from complex version of (\ref{jack22}). This gives Corollary~\ref{sol1} with $S_{p} = C p' M^{2}_{p'}$ where $C$ is some universal constant. 
%\end{remark} 

\begin{remark}
It is well-known (see \cite{EI2}) that for every $p>1$, there exists a finite constant $K_{p}>1$ such that for any $d\geq1$, all polynomials $f\in\mathcal{P}^{\leq d}$ satisfy the moment comparison
\begin{equation} \label{eq:moments}
\|f\|_{L_p(d\gamma_n)} \leq K_{p}^d \|f\|_{L_1(d\gamma_n)}.
\end{equation}
It follows from \cite[Theorem~5]{CE20} that for any $p>1$,
\begin{equation} \label{eq:ce}
\|L^{-1}f\|_{L_p(d\gamma_n)} \leq C_p \frac{\|f\|_{L_p(d\gamma_n)}}{1+\log\big(\|f\|_{L_p(d\gamma_n)}\big/\|f\|_{L_1(d\gamma_n)}\big)}.
\end{equation}
Inequality \eqref{eq:ce} implies that Mendel and Naor's question \eqref{conj2} has an affirmative answer for those functions satisfying the converse of \eqref{eq:moments}, i.e., $\|f\|_{L_p(d\gamma_n)} \geq K_{p}^d \|f\|_{L_1(d\gamma_n)}$ for some $K_p>1$. While this condition is not true for all $f\in\mathcal{P}^{\geq d}$, \eqref{eq:ce} confirms \eqref{conj2} for a conceptually different class of functions $f$. We refer to \cite{CE20} for further information on \eqref{eq:ce}, including versions on the Hamming cube and vector-valued extensions.
\end{remark}

\subsection{Proof of Theorem~\ref{anal1}}
Recall that a polynomial $f : \mathbb{R}^{2} \to \mathbb{C}$ is called analytic if  $f(x_{1}, y_{1}) = g(x_{1}+iy_{1})$ for some polynomial $g : \mathbb{R} \to \mathbb{C}$. Similarly analytic polynomials $f : \mathbb{R}^{2n} \to \mathbb{C}$ can be written as 
\begin{align*}
f(z) = \sum_{|\alpha| \leq d} c_{\alpha} z^{\alpha}
\end{align*}
for some $d\geq 0$, where $z^{\alpha} \stackrel{\mathrm{def}}{=} \prod_{j =1}^{n} z_j^{\alpha_{j}}$ for some multi-index $\alpha \in (\mathbb{Z}_{+})^{n}$ and $z=(z_{1}, \ldots, z_{n}) \in \mathbb{C}^{n}$, where $z_{j} = z_{j}+i y_{j}$. In what follows, we will write $f = f(z_{1}, \ldots, z_{n})$ (or simply $f(z)$) instead of $f(x_{1}, y_{1}, \ldots, x_{n}, y_{n})$ for an analytic polynomial $f$.  For $z=(z_{1}, \ldots,z_{n}) \in \mathbb{C}^n$ and $w \in \mathbb{C}$ we set $wz = (wz_{1}, \ldots, wz_{n})$.

\begin{lemma}[Heat-smoothing for analytic polynomials]
For any finite $p\geq 1$ we have 
\begin{align}\label{heatsm}
\left\| \sum_{|\alpha|\geq d} e^{-t|\alpha|} c_{\alpha}  z^{\alpha}\right\|_{L^{p}(d\gamma_{2n})} \leq  e^{-t d} \left\| \sum_{|\alpha|\geq d} c_{\alpha} z^{\alpha}\right\|_{L^{p}(d\gamma_{2n})}
\end{align}
for all $d\geq 0$, all $t \geq 0$ and all  finite sums $\sum_{|\alpha|\geq d} c_{\alpha} z^{\alpha}$. 
\end{lemma}
\begin{proof}
Fix $t\geq 0$ and any integer $m \geq d$.  Let $\mathbb{T}$ be the unit circle on the complex plane and let $K$ be the linear functional on the subspace $\mathrm{span}_{\mathbb{C}}\{w^{d}, w^{d+1}, \ldots, w^{m}\} \subseteq C(\mathbb{T})$ which is defined by
\begin{align*}
K(z^{k}) = e^{-tk} \quad \text{for} \quad k=d,\ldots, m.
\end{align*}
Clearly $K(p(w)) = p(e^{-t})$ for all $p(w) = \sum_{k=d}^{m} a_{k} w^{k}$. By the Hahn--Banach theorem, we can extend $K$ to $\widetilde{K}$ defined on $C(\mathbb{T})$ so that 
\begin{align*}
|\widetilde{K}(g)| \leq C(t,d) \|g\|_{C(\mathbb{T})} \quad \text{for all} \quad g \in C(\mathbb{T}),
\end{align*}
where $C(t,d)$ is the smallest nonnegative constant such that 
\begin{align}\label{max1}
|p(e^{-t})| \leq C(t,d) \| p\|_{C(\mathbb{T})} \quad \text{for all} \quad p(w) = \sum_{k=d}^{m} a_{k} w^{k}.
\end{align}

To find $C(t,d)$, fix any polynomial $p(w)$ of the above form and consider the analytic function $h(w) = \frac{p(w)}{w^{d}}$ defined in the closed unit disk $\mathbb{D} = \{w \in \mathbb{C}: |w|\leq 1 \}$. The inequality $\sup_{w \in \mathbb{T}}|h(w)| \leq \| p\|_{C(\mathbb{T})}$ and the maximum principle imply that $|p(w)|\leq |w|^{d}\| p\|_{C(\mathbb{T})}$ for all $w \in \mathbb{D}$. In particular $|p(e^{-t})| \leq e^{-td} \| p\|_{C(\mathbb{T})}$, i.e., the inequality \eqref{max1} holds with $C(t,d) = e^{-td}$. By the Riesz representation theorem there exists a complex-valued measure $d\mu$ on $\mathbb{T}$ such that $\widetilde{K}(g) = \int_{\mathbb{T}} g(w) d\mu(w)$ and $\|\mu\|_{TV} = e^{-td}$, where $\|\mu\|_{TV}$ denotes the total variation norm of $d\mu$. 

Fix a polynomial of the form $\sum_{|\alpha|\geq d} c_{\alpha}z^{\alpha}$ and pick $m$ such that $c_\alpha=0$ for $|\alpha|>m$. Then, since $\widetilde{K}$ extends $K$, we have 
\begin{align*}
\left\| \sum_{|\alpha|\geq d} e^{-t|\alpha|}  c_{\alpha}  z^{\alpha}\right\|_{L^{p}(d\gamma_{2n})}&
 = 
\left\| \sum_{d\leq|\alpha|\leq m} \int_{\mathbb{T}}w^{|\alpha|} c_{\alpha}  z^{\alpha} d\mu(w)\right\|_{L^{p}(d\gamma_{2n})}\\
&\leq\int_{\mathbb{T}} \left\| \sum_{|\alpha|\geq d} c_{\alpha}  (wz)^{\alpha} \right\|_{L^{p}(d\gamma_{2n})}d|\mu|(w) = e^{-t d} \left\| \sum_{|\alpha|\geq d} c_{\alpha} z^{\alpha}\right\|_{L^{p}(d\gamma_{2n})},
\end{align*}
where the last equality follows from the (complex) rotational invariance of the Gaussian measure and the estimate 
$\|\mu\|_{TV} = e^{-td}$.
\end{proof}

To prove  Theorem~\ref{anal1} we integrate \eqref{heatsm} in $t$ over the ray $[0, \infty)$ to obtain  
\begin{equation*}
\begin{split}
\left\| \sum_{|\alpha|\geq d}  \frac{1}{|\alpha|}c_{\alpha} z^{\alpha}\right\|_{L^{p}(d\gamma_{2n})} & = \left\| \int_0^\infty \sum_{|\alpha|\geq d} e^{-t|\alpha|} c_\alpha z^\alpha \ dt\right\|_{L_p(d\gamma_{2n})}
\\ &  \leq \int_{0}^{\infty} \left\| \sum_{|\alpha|\geq d} e^{-t|\alpha|}c_{\alpha}  z^{\alpha}\right\|_{L^{p}(d\gamma_{2n})}  dt 
 \stackrel{\eqref{heatsm}}{\leq}  \frac{1}{d}\left\| \sum_{|\alpha|\geq d} c_{\alpha}  z^{\alpha}\right\|_{L^{p}(d\gamma_{2n})}.
\end{split}
\end{equation*}
Next, notice that a direct computation shows $L z^{\alpha} = |\alpha| z^{\alpha}$. In particular, for a polynomial  $f = \sum_{|\alpha|\geq d} c_{\alpha} z^{\alpha}$, the last inequality (after rescaling the coefficients) implies $d\|f\|_{L^{p}(d\gamma_{2n})} \leq \| L f\|_{L^{p}(d\gamma_{2n})}$. This finishes the proof of \eqref{naoran}. To prove \eqref{gradan}, notice that we can write  \hfill$\Box$
\begin{equation*}
\begin{split}
\Bigg\| \sum_{|\alpha|\geq d}  \frac{1}{|\alpha|^{1/2}}c_{\alpha} &z^{\alpha}\Bigg\|_{L^{p}(d\gamma_{2n})}  = \frac{1}{\sqrt{\pi}} \left\| \int_0^\infty \sum_{|\alpha|\geq d} e^{-t|\alpha|} c_\alpha z^\alpha \ \frac{dt}{\sqrt{t}}\right\|_{L_p(d\gamma_{2n})}
\\ &  \leq \frac{1}{\sqrt{\pi}} \int_{0}^{\infty} \left\| \sum_{|\alpha|\geq d} e^{-t|\alpha|}c_{\alpha}  z^{\alpha}\right\|_{L^{p}(d\gamma_{2n})}  \frac{dt}{\sqrt{t}} 
 \stackrel{\eqref{heatsm}}{\leq}  \frac{1}{\sqrt{d}}\left\| \sum_{|\alpha|\geq d} c_{\alpha}  z^{\alpha}\right\|_{L^{p}(d\gamma_{2n})},
\end{split}
\end{equation*}
which is equivalent to $\sqrt{d}\|f\|_{L^{p}(d\gamma_{2n})} \leq \| L^{1/2} f\|_{L^{p}(d\gamma_{2n})}$. The desired inequality \eqref{gradan} now readily follows from \eqref{riesz} with $K_p = \tfrac{1}{m_p}$.

\subsection{Proof of Theorem~\ref{analgrad}} The approach will be similar to that of the previous section. We start by showing the corresponding (sharp) statement for $L$.
\begin{lemma}\label{bern1}
For any $p \geq 1$ and $d\geq 0$, we have 
\begin{align}\label{bern2}
\| Lf \|_{L^{p}(d\gamma_{2n})} \leq d \| f\|_{L^{p}(d\gamma_{2n})}
\end{align}
for any analytic  polynomial $f \in \mathcal{P}^{\leq d}$ on $\mathbb{C}^{n}$. 
\end{lemma}
\begin{proof}
Let $K$ be the linear functional on the subspace $\mathrm{span}_{\mathbb{C}}\{1, w, \ldots, w^{d}\} \subseteq C(\mathbb{T})$ given by
\begin{align*}
K(w^{k}) = k \quad \text{for all} \quad k=0,1,\ldots, d.
\end{align*}
Clearly $K(p(w)) = p'(1)$ for all $p(w) = \sum_{k=0}^{d} a_{k} w^{k}$. By the Hahn--Banach theorem and the Riesz representation theorem there exists complex-valued measure $d\mu$ on $\mathbb{T}$ such that $\int_{\mathbb{T}} p(w) d\mu(w) = p'(1)$ for all $p = \sum_{k=0}^{d} a_{k} w^{k}$ and $\|\mu\|_{TV}$ is the smallest constant such that 
\begin{align}\label{mark1}
|p'(1)| \leq \| \mu \|_{TV} \|p\|_{C(\mathbb{T})} \quad \text{for all} \quad p = \sum_{k=0}^{d} a_{k} w^{k}.
\end{align}
By Bernstein's inequality for trigonometric polynomials we can choose   $\| \mu \|_{TV} = d$ in \eqref{mark1} and this is the best possible constant as the example $p(w)=w^{d}$ shows. 

Next, pick an arbitrary analytic polynomial $f(z)  = \sum_{|\alpha|\leq d} c_{\alpha}z^{\alpha}$ of degree at most $d$. As in the previous section we can write 
\begin{align*}
\|L f \|_{L^{p}(d\gamma_{2n})} = \left\| \int_{\mathbb{T}} f(wz) d\mu(w)\right\|_{L^{p}(d\gamma_{2n})}\leq \int_{\mathbb{T}} \left\|  f(wz) \right\|_{L^{p}(d\gamma_{2n})}d|\mu|(w) = d \|f\|_{L^{p}(d\gamma_{2n})}
\end{align*}
and the proof of the lemma is complete.
\end{proof}

To prove Theorem~\ref{analgrad}  we use the estimate 
\begin{align}\label{lust}
\| L^{1/2} f\|_{L^{p}(d\gamma_{2n})} \leq 2 \| f\|^{1/2}_{L^{p}(d\gamma_{2n})} \| Lf\|^{1/2}_{L^{p}(d\gamma_{2n})},
\end{align}
which is valid for all polynomials $f$ (see \cite[Lemma 19]{EI1}).  Then boundedness of the Riesz transform  \eqref{riesz} combined with estimate  \eqref{lust}  and Lemma~\ref{bern1} imply
\begin{align*}
\|\nabla f\|_{L^{p}(d\gamma_{2n})} &\stackrel{\eqref{riesz}}{\leq} M_{p}  \| L^{1/2} f\|_{L^{p}(d\gamma_{2n})} \\
&\stackrel{\eqref{lust}}{\leq} 2 M_{p} \| f\|^{1/2}_{L^{p}(d\gamma_{2n})} \| Lf\|^{1/2}_{L^{p}(d\gamma_{2n})} \stackrel{\eqref{bern2}}{\leq} 2M_{p} \sqrt{d} \|f\|_{L^{p}(d\gamma_{2n})}.
\end{align*}
This finishes the proof of Theorem~\ref{analgrad}. \hfill$\Box$

\subsection{Proof of Theorem~\ref{thm:mom}} Fix an analytic polynomial $f\in\mathcal{P}^{\leq d}$ on $\mathbb{C}^n$. Then, inequality \eqref{bern2} implies that for every $t\geq0$ and $q\geq1$,
\begin{equation} \label{eq:revheat}
\begin{split}
\|f\|_{L^q(d\gamma_{2n})}& = \left\| \sum_{m=0}^\infty \frac{(tL)^m e^{-tL} f}{m!} \right\|_{L^q(d\gamma_{2n})} \leq \sum_{m=0}^\infty \frac{t^m\|L^me^{-tL}f\|_{L^q(d\gamma_{2n})}}{m!} 
\\ & \stackrel{\eqref{bern2}}{\leq} \sum_{m=0}^\infty \frac{(td)^m \|e^{-tL}f\|_{L^q(d\gamma_{2n})}}{m!} = e^{td} \|e^{-tL}f\|_{L^q(d\gamma_{2n})}.
\end{split}
\end{equation}
Now recall that $[e^{tL}f](w) = f(e^tw)$ for every $t\geq0$ and $w\in\mathbb{C}$. Therefore, \eqref{eq:revheat} can be equivalently rewritten as 
\begin{equation} \label{prejan}
\|f\|_{L^q(d\gamma_{2n})} \leq \frac{1}{\rho^d} \left( \int_{\mathbb{C}^n} |f(\rho z)|^q \ d\gamma_{2n}(z)\right)^{1/q},
\end{equation}
where $\rho\in(0,1]$. A classical result of Janson \cite[Theorem~11]{Jan83} asserts that for every $|\rho|\leq\sqrt{\tfrac{p}{q}}$ and any analytic function $f:\mathbb{C}^n\to\mathbb{C}$,
\begin{equation} \label{janson}
\left( \int_{\mathbb{C}^n} |f(\rho z)|^q \ d\gamma_{2n}(z)\right)^{1/q} \leq \|f\|_{L^p(d\gamma_{2n})}
\end{equation}
and the conclusion of the theorem follows by combining \eqref{prejan} and \eqref{janson}.

\begin{remark} \label{rem:mom}
The simple proof of Theorem \ref{thm:mom} shows that a Bernstein-type inequality in the spirit of \eqref{bern2} for the generator of a hypercontractive semigroup formally implies moment comparison results\footnote{Such estimates are often referred to as Nikol'skii-type inequalities in approximation theory.} such as \eqref{eq:mom}. A simple variant of our argument readily implies a result of Defant and Masty\l o \cite[Theorem~2.1]{DM16}, who proved the optimal $L^p-L^q$ moment comparison for analytic polynomials on the torus $\mathbb{T}^n$, under the additional assumption that $q\geq1$. Indeed, to deduce this result one has to repeat the preceeding argument verbatim with the exception that a result of Weissler \cite[Corollary~2.1]{Wei80} has to be used instead of Janson's inequality \eqref{janson}.
\end{remark}

 %To the best of our knowledge, it remains unknown whether for every $1<p<q$ and all polynomials $f\in\mathcal{P}^{\leq d}$ on $\mathbb{R}^n$, we have (perhaps up to a universal constant)
 %\begin{equation}
 %\| f\|_{L^{q}(d\gamma_{n})} \leq \left(\frac{q-1}{p-1}\right)^{d/2} \| f\|_{L^{p}(d\gamma_{n})}.
 %\end{equation}
 
\bibliographystyle{siam}
\bibliography{ReverseV3}
%\nocite{*}

\end{document}